\documentclass[reqno]{amsart}
\usepackage[utf8]{inputenc}
\usepackage{xcolor,fullpage,enumitem}
\usepackage{tikz,amsmath,amssymb,color,mathtools}
\usetikzlibrary{calc}
\usepackage{float}
\usepackage{pgf}
\usepackage{amsrefs}
\usepackage[normalem]{ulem}
\usepackage{graphicx}
\usepackage{subcaption}
\usepackage{hyperref}

\newtheorem{theorem}{Theorem}[section]

\newtheorem{lemma}[theorem]{Lemma}
\theoremstyle{definition}

\theoremstyle{definition}

\theoremstyle{remark}
\newtheorem{remark}[theorem]{Remark}

\DeclarePairedDelimiter\floor{\lfloor}{\rfloor}
\newcommand\E{\operatorname{\mathbb E}{}}
\newcommand\Xn{{X_1(n)}}
\newcommand\Yn{{X_2(n)}}
\newcommand\Catalan{c}
\newcommand\Cn{C_n}
\newcommand\Var{\operatorname{Var}}
\renewcommand\P{\operatorname{\mathbb P{}}}
\newcommand\dto{\overset{\mathrm{d}}{\longrightarrow}}
\newcommand\pto{\overset{\mathrm{p}}{\longrightarrow}}
\newcommand\ntoo{\ensuremath{{n\to\infty}}}
\newcommand\nii{\floor{n^{2/3}}}
\newcommand\cX{{\mathcal X}}
\newcommand\cE{{\mathcal E}}
\renewcommand\le{\leqslant}
\renewcommand\ge{\geqslant}
\renewcommand\leq{\leqslant}

\usepackage{thmtools} 
\usepackage{thm-restate}

\title{How many coin tosses would you need until you get \\ $n$ Heads or $m$ Tails?}

\author[Janson]{Svante Janson}
\address[S.~Janson]{Department of Mathematics, Uppsala University, PO Box 480, SE-751~06 Uppsala, Sweden}
\email{\textcolor{blue}{\href{mailto:svante.janson@math.uu.se}{svante.janson@math.uu.se}}}

\author[Martinez]{Lucy Martinez}
\address[L.~Martinez \& D.~Zeilberger]{Department of Mathematics, Rutgers University, Piscataway, NJ 08854}
\email{\textcolor{blue}{\href{mailto:lucy.martinez@rutgers.edu}{lucy.martinez@rutgers.edu}}}

\author[Zeilberger]{Doron Zeilberger}
\email{\textcolor{blue}{\href{mailto:doronzeil@gmail.com}{doronzeil@gmail.com}}}

\begin{document}

\begin{abstract}
We harness both human ingenuity and the power of symbolic computation to study the number of coin tosses until reaching $n$ Heads or $m$ Tails. We also talk about the closely related problem of reaching $n$ Heads and $m$ Tails. This paper is accompanied by a Maple package that enables fast computation of expectations, variances, and higher moments of these quantities.
\end{abstract}

\maketitle

\section{Preface}\label{Spreface}

If you toss a coin whose probability of Heads is $p$, until you reach $n$ Heads, you should expect to make $n/p$ coin tosses, and the variance and higher moments are easily derived from the explicit probability generating function, (as usual $q\coloneqq1-p$)
\begin{align}\label{negbin}
\sum_{k=0}^{\infty} \binom{n+k-1}{n-1}(px)^{n} (qx)^{k} \, = \, \left
  (\frac{px}{1-qx} \right )^n,
\end{align}
which is essentially the \emph{negative-binomial distribution}~\cite{W} (note that usually one only counts the number of Tails until you reach $n$ Heads, but we are interested in the total number of coin-tosses, so we add the $n$ Heads). From this probability generating function we can extract not only the expectation, $n/p$, but also the variance $\frac{n(1-p)}{p^2}$, and by repeated differentiation with respect to $x$,
and plugging in $x=1$, we can easily derive explicit expressions of as many as desired \emph{factorial moments}, that in turn, yield the \emph{moments}, and from them the \emph{central moments}. Then we can compute the \emph{scaled central moments}, take the limit as $n \rightarrow \infty$ and prove that for a {\bf fixed} $p$ it tends to the good old Normal Distribution. Of course, in this simple case we can also derive a \emph{local limit law}.
(For a probabilist, these are examples of the classical central limit theorem and local limit theorem, see e.g.~\cite[Theorems 7.1.1 and 7.7.6]{Gut}.)

But what if you are not a \emph{Headist}? What if you like Tails just as much, and stop as soon as you get
$n$ Heads OR $m$ Tails? Another interesting stopping rule is to make {\bf both} Heads and Tails happy and keep tossing until you get $n$ Heads AND $m$ Tails. Now things are not as nice and simple. Nevertheless, using Wilf-Zeilberger algorithmic proof theory~\cite{PWZ}, we can derive the next-best thing, linear recurrences that enable very fast computation of these quantities.
These will be presented in Section~\ref{sec:recurrences}.

In the special case of a \emph{fair} coin, and the same desired number of Heads and Tails (let's call it $n$), we get, {\bf explicit} expressions not only for the expectation and variance, but for as many as-desired moments (we went up to $200^{th}$, but could go much further). Then we (or rather Maple) can compute the limits (as $n\to \infty$) of the scaled moments, and {\bf surprise!} they coincide {\bf exactly} with the central-scaled moments of $-|N(0,1)|$, the continuous probability distribution whose \emph{probability density function} (pdf) is
\[
\frac{e^{-x^2/2}}{\sqrt{\pi/2}},
\]
supported in $-\infty < x < 0$. This will be accomplished in Section~\ref{sec:Explicit-exprs}.

While we (or rather our computer) can prove this convergence for the first $200$ moments, and with a larger computer, the first $2000$, we can {\bf not} prove it for \emph{all} moments.
In Section~\ref{Shuman}, we will prove it completely, using purely human-generated, paper-and-pencil mathematics.

The reader must have noticed that our problem brings to mind the {\bf first-ever} probability problem- the one that gave rise to the theory of probability, namely the ``Problem of the Points''\cite{W2}. But Fermat and Pascal were only interested in the probability of getting $n$ Heads vs.\ getting $m$ Tails, {\bf not} about the duration. A literature search for ``Problem of the points'' and ``duration'' lead to only one hit~\cite{RSW} (yes, the same Ross who wrote the famous textbook), but it has measure zero intersection with the present paper. In the setting $m=n$, corresponding to the case of $n$ Heads OR $n$ Tails, related results have been obtained by Volkov and Wiktorsson~\cite{VW}. Their work concerns an expectation similar to, but distinct from, the one studied here. Additional discussion appears at the end of Section~\ref{sec:recurrences}.

\subsection{The Maple package} The file {\tt CoinToss.txt} is freely available from

{\tt https://sites.math.rutgers.edu/\~{}zeilberg/tokhniot/CoinToss.txt}, 

\noindent and allows you to experiment with these quantities. We include some plots in Figure~\ref{fig: expected tosses} showing the expected number of tosses as a function of the probability of getting Heads under various goals. We also include two plots in Figure~\ref{fig: pdf} of the probability mass function for the (scaled) discrete random variable representing the number of coin tosses until a loaded coin reaches either 200 Heads OR 200 Tails, where with probability $p$ you get Heads.
\begin{figure}[H]
    \centering
   \begin{subfigure}[b]{0.32\textwidth}
      \centering
     \includegraphics[height=2in]{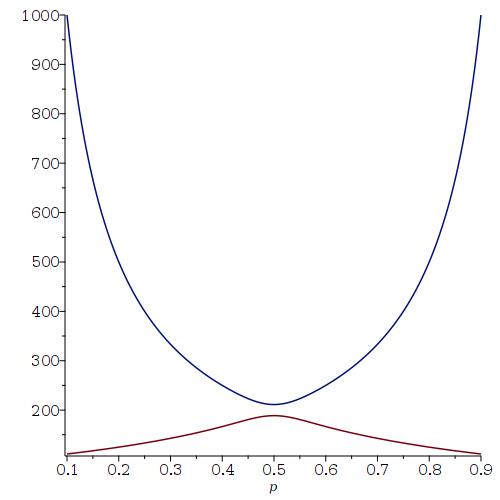}
    \caption{In this figure, the goal is to get 100 Heads AND (top) / OR (bottom) 100 Tails.}
  \end{subfigure}
  \begin{subfigure}[b]{0.32\textwidth}
      \centering
      \includegraphics[height=2in]{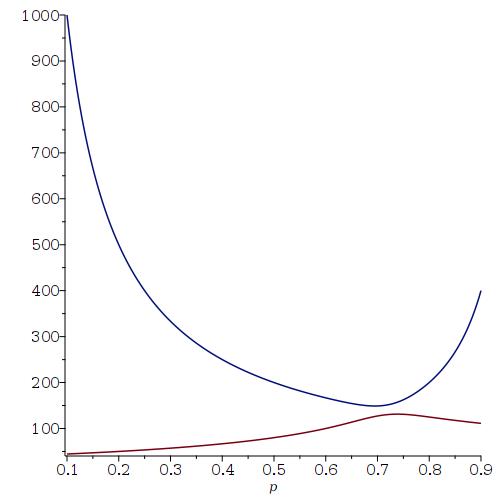}
       \caption{In this figure, the goal is to get 100 Heads AND (top) / OR (bottom) 40 Tails.}
   \end{subfigure}
  \begin{subfigure}[b]{0.32\textwidth}
       \centering
       \includegraphics[height=2in]{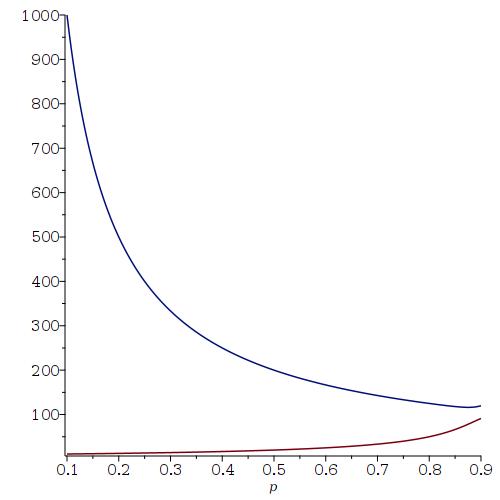}
       \caption{In this figure, the goal is to get 100 Heads AND (top) / OR (bottom) 10 Tails.}
   \end{subfigure}
   \caption{Each of the sub-figures shows the plots of the expected number of tosses as a function of the probability of getting Heads (from 0.1 to 0.9).}
    \label{fig: expected tosses}
\end{figure}

\begin{figure}[H]
    \centering
   \begin{subfigure}[b]{0.32\textwidth}
      \centering
     \includegraphics[height=2in]{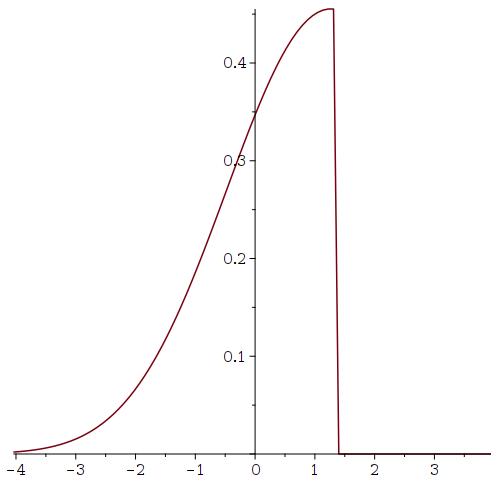}
    \caption{In this figure, the probability of getting Heads is $p=\frac{1}{2}$.}
  \end{subfigure}
  \hspace{0.5 in}
  \begin{subfigure}[b]{0.32\textwidth}
      \centering
      \includegraphics[height=2in]{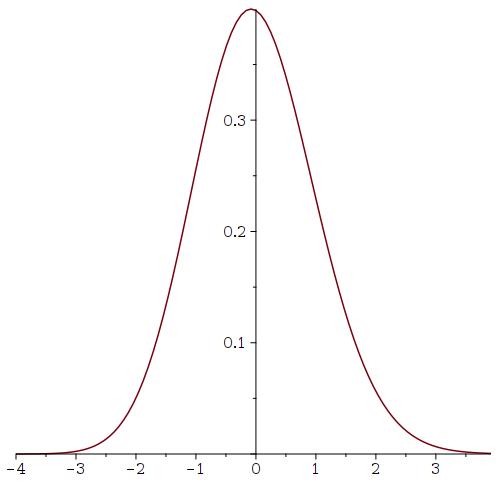}
       \caption{In this figure, the probability of getting Heads is $p=\frac{1}{3}$.}
   \end{subfigure}
   \caption{Each of the sub-figures shows the plots of the probability mass function for the (scaled) discrete random variable representing the number of coin tosses until a loaded coin reaches either 200 Heads OR 200 Tails, where with probability $p$ you get Heads.}
    \label{fig: pdf}
\end{figure}

\section{Recurrences for the Duration with a Loaded Coin and Different Target Goals}\label{sec:recurrences}

You are tossing a coin whose probability of Heads is $p$ (and hence the probability of Tails is $q\coloneqq1-p$). We consider two random variables
\begin{itemize}
    \item $X_1(n,m; p)$: The number of tosses until reaching (for the first time) either $n$ Heads OR $m$ Tails, and
    \item $X_2(n,m;p)$: The number of tosses until reaching (for the first time) $n$ Heads AND $m$ Tails.
\end{itemize}

The {\bf probability generating function} of $X_1$, in $x$, let's call it $F_1(n,m;p)(x)$ is:
\[
F_1(n,m;p)(x)=
(qx)^m \sum_{h=0}^{n-1} \binom{h+m-1}{m-1}(px)^h + (px)^n \sum_{t=0}^{m-1} \binom{t+n-1}{n-1}(qx)^t.
\]

On the other hand, that of $X_2(n,m;p)$, let's call it $F_2(n,m;p)(x)$ is:
\[
F_2(n,m;p)(x)=
(qx)^m \sum_{h=n}^{\infty} \binom{h+m-1}{m-1} (px)^h + (px)^n \sum_{t=m}^{\infty} \binom{t+n-1}{n-1} (qx)^t.
\]

Unlike the probability generating function for the Negative Binomial distribution~\cite{W},
$F_1(n,m;p)(x)$ and $F_2(n,m;p)(x)$ do not have {\bf closed-form},
but thanks to {\bf Wilf-Zeilberger Algorithmic Proof theory}~\cites{PWZ, Z}, they have the next-best thing, linear recurrences with polynomial coefficients (that happen to be third-order), in each of $n$ and $m$, that enable an efficient compilation of a table of these. We observe that $F_1+F_2$ {\it does} have a closed-form,
\begin{equation}
  \label{eq:sj1}
F_1(n,m;p)(x)+F_2(n,m;p)(x)= \left(\frac{qx}{1-px}\right)^m  + \left(\frac{px}{1-qx}\right)^n.  
\end{equation}
(This reflects that fact that for each sequence of coin tosses, the two random
variables $X_1(n,m;p)$ and $X_2(n,m;p)$ equal, in some order, the numbers of tosses required to reach $n$ Heads and to reach $m$ Tails.)
The actual recurrences are too complicated to reproduce here but can be looked up from the output file, {\tt https://sites.math.rutgers.edu/\~{}zeilberg/tokhniot/oCoinToss2.txt}.

Let $L_1(n,m;p)$ be the expectation of $X_1(n,m;p)$ and let $L_2(n,m;p)$ be the expectation of $X_2(n,m;p)$, then both satisfy the {\bf same} system of third-order linear pure recurrences. We have the following pure recurrences, in $n$ and $m$, respectively; the two recurrences are equivalent by interchanging Heads and Tails
(below $L(n,m)$ stands for either $L_1(n,m;p)$ and $L_2(n,m;p)$);
\begin{align*}
    L(n, m) =&\frac{\left(p n +p m -2 p +2 n -2\right) }{n -1}  \cdot L\left(n -1, m\right) \\
    &-\frac{\left(2 p n +2 p m -4 p +n -1\right) }{n -1}\cdot L \! \left(n -2, m\right) +\frac{p \left(m -2+n \right) }{n -1} \cdot L \! \left(n -3, m\right),\\
    L(n , m)= &-\frac{\left(p n +p m -2 p -n -3 m +4\right) }{m -1} \cdot L \! \left(n , m -1\right)\\
    &+\frac{\left(2 p n +2 p m -4 p -2 n -3 m +5\right) }{m -1} \cdot L \! \left(n , m -2\right)
-\frac{\left(p-1\right) \left(m -2+n \right) }{m -1} \cdot L \! \left(n , m -3\right).
\end{align*}
Of course $L_1(n,m;p)$ and $L_2(n,m;p)$ differ in the \emph{initial conditions}. Here they are:
\begin{align*}
    &[[L_1(1,1),L_1(1,2),L_1(1,3)], [L_1(2,1),L_1(2,2),L_1(2,3)], [L_1(3,1),L_1(3,2),L_1(3,3)] ] = \\
    & [[1, -p +2, p^{2}-3 p +3],
    [p +1, -2 p^{2}+2 p +2, 3 p^{3}-7 p^{2}+3 p +3], \\
    & [p^{2}+p +1, -3 p^{3}+2 p^{2}+2 p +2, 6 p^{4}-12 p^{3}+3 p^{2}+3 p +3] ],
\end{align*}
and
\begin{align*}
    &[[L_2(1,1),L_2(1,2),L_2(1,3)], [L_2(2,1),L_2(2,2),L_2(2,3)], [L_2(3,1),L_2(3,2),L_2(3,3)]] =\\
    & \Biggl[\left[\frac{p^{2}-p +1}{p \left(1-p\right)}, -\frac{p^{3}-3 p^{2}+p -1}{p \left(1-p\right)}, \frac{p^{4}-4 p^{3}+6 p^{2}-p +1}{p \left(1-p\right)}\right], \\
    & \left[\frac{p^{3}-2 p +2}{p \left(1-p\right)}, -\frac{2 p^{4}-4 p^{3}+2 p -2}{p \left(1-p\right)}, \frac{3 p^{5}-10 p^{4}+10 p^{3}-2 p +2}{p \left(1-p\right)}\right], \\
    &\left[\frac{p^{4}-3 p +3}{p \left(1-p\right)}, -\frac{3 p^{5}-5 p^{4}+3 p -3}{p \left(1-p\right)}, \frac{3 \left(2 p^{6}-6 p^{5}+5 p^{4}-p +1\right)}{p \left(1-p\right)} \right] \Biggr].
\end{align*}

These recurrences are implemented in procedures {\tt fAveF(n,m,p)} and {\tt FaveF(n,m,p)} respectively. For example to find the expected number of coin-tosses it takes if you toss a loaded coin whose probability of Heads is $\frac{1}{3}$ until it reaches, for the first time $100\,i$ Heads OR $200\,i$ Tails, for $1 \leq i \leq 7$, type:

\begin{itemize}
    \item[] {\tt restart: read `CoinToss.txt`:t0:=time(): }
    \item[] {\tt evalf([seq(faveF(100*i,200*i,1/3),i=1..7)]);time()-t0; } 
\end{itemize}
getting
\[ [285.3561686, 579.2804255, 874.6196952, 1170.690974, 1467.229920, 1764.101012, 2061.223764], \]
and it took $0.563$ seconds.

For comparison, if you do it directly, {\bf not} using the recurrence, but rather the definition as a sum, typing
\begin{itemize}
    \item[] {\tt  restart: read `CoinToss.txt`: t0:=time(): }
    \item[] {\tt evalf([seq(fave(100*i,200*i,1/3),i=1..7)]);time()-t0; } 
\end{itemize}
you would get the same output, but it took more than $12$ seconds.

What about $L_2(n,m;p)$, i.e. $L_2(100i,200i;\frac{1}{3})$? Type:
\begin{itemize}
    \item[] {\tt restart: read `CoinToss.txt`: }
    \item[] {\tt t0:=time():evalf([seq(FaveF(100*i,200*i,1/3),i=1..7)]);time()-t0; }
\end{itemize}
getting
\[ [314.6438314, 620.7195745, 925.3803048, 1229.309026, 1532.770080, 1835.898988, 2138.776236], \]
and this took $0.561$ seconds, and the direct way took more than $12$ seconds.

We observe that Equation~\eqref{eq:sj1} implies the relation
\begin{align}\label{l1l2}
  L_1(n,m;p)+L_2(n,m;p) = \frac{n}{p}+\frac{m}{1-p}.
\end{align}
We also observe that for positive integers $a$ and $b$ we have the {\bf explicit} expressions,
\begin{align}\label{l1ab}
L_1(an,bn; \frac{a}{a+b}) &= (a+b)n\left(1-\frac{((a+b)n)!}{(an)!(bn)!} \cdot \left( \frac{a^a b^b}{(a+b)^{a+b}} \right)^n \right), \\\label{l2ab}
L_2(an,bn; \frac{a}{a+b})&= (a+b)n\left(
1+\frac{((a+b)n)!}{(an)!(bn)!} \cdot \left( \frac{a^a b^b}{(a+b)^{a+b}} \right)^n \right).
\end{align}
These are asymptotically
\[
(a+b)n\left(1 \pm \sqrt{ \frac{a+b}{2ab\pi}} \cdot \frac{1}{\sqrt{n}} \right).
\]

\begin{remark}
For any specific positive integers $a$ and $b$, the expressions for $L_1(an,bn; \frac{a}{a+b})$ and $L_2(an,bn; \frac{a}{a+b})$ are routinely provable using WZ algorithmic proof theory, but we are unable to prove it automatically, i.e. using algorithmic proof theory, for all positive integers $a$ and $b$. However, using ``human ingenuity'' this can be proved in general. We present a proof in Section~\ref{Spfl1ab}.
\end{remark}

So $L_1(an,bn; \frac{a}{a+b})/((a+b)(n))$ and $L_2(an,bn; \frac{a}{a+b})/((a+b)(n))$ converge slowly (as $n^{-1/2}$) to $1$ as $n$ goes to infinity. On the other hand if $p>\frac{a}{a+b}$ then $L_1(an,bn;p)/n$ and $L_2(an,bn;p)/n$ converge exponentially fast to $bp$,
and if $p<\frac{a}{a+b}$ then they converge exponentially fast to $ap$. This makes sense, since when a coin is loaded in favor of achieving your goal you should expect to achieve your goal only a bit later than if the other side of the coin didn't matter.

\begin{remark}\label{RVolkov}
Volkov and Wiktorsson~\cite{VW} recently studied some related aspects of the case ``$n$ Heads or $n$ Tails'' (thus with $m=n$).
In particular, they study~\cite[Theorem 2.1]{VW} the expectation of (number of Heads - number of Tails) when we stop having reached our goal after $X_1(n,n; p)$ tosses. (This is $S_{\Xn}$ in the notation of Section~\ref{Shuman} below.) 
By Wald's identity~\cite[Theorem 10.14.3]{Gut}, this expectation equals
$(p-q)L_1(n,n;p)$, and thus the formula in~\cite[Theorem 2.1]{VW} is
equivalent to
\begin{align}\label{vw}
  L_1(n,n;p)=n\sum_{j=0}^{n-1} \Catalan_j (pq)^j,
\end{align}
where 
\begin{align}\label{Catalan}
\Catalan_j\coloneqq\frac{(2j)!}{j!\,(j+1)!},  
\end{align}
are the Catalan numbers.
In the fair case $p=\frac12$, it is easily verified that Equation~\eqref{vw} agrees
with the expression given by Equation~\eqref{l1ab} in the special case when $a=b=1$,
\begin{align}\label{vw2}
L_1(n,n;\tfrac12)
=  2n \left(1-\frac{(2n)!}{n!^2}\cdot 4^{-n} \right).
\end{align}
\end{remark}

\section{Explicit expressions for the moments of the number of tosses until getting \\ $n$ Heads or $n$ Tails with a fair coin}\label{sec:Explicit-exprs}

The probability generating function for the number of tosses until a fair coin reaches $n$ Heads OR $n$ Tails is
\begin{align*}
(\tfrac12 x)^n \sum_{h=0}^{n-1} \binom{h+n-1}{n-1}(\tfrac12x)^h + (\tfrac12 x)^n \sum_{t=0}^{n-1} \binom{t+n-1}{n-1}(\tfrac12 x)^t=(\tfrac12)^{n-1} \sum_{h=0}^{n-1} \binom{h+n-1}{n-1}(\tfrac12)^{h}x^{h+n}.
\end{align*}

Recall that the $r$-th {\bf factorial moment} of a random variable $X$ is
\[
\E[X(X-1)\dots (X-r+1)]= r!\E\left[\binom{X}{r}\right] .
\]

Let $A(n,r)$ be the the $r$-th factorial moment of our random variable $X=X_1(n,n;\tfrac12)$
(number of tosses of a fair coin until you get for the first time $n$ Heads OR $n$ Tails). We have,
\[
A(n,r)=(\tfrac12)^{n-1} \sum_{h=0}^{n-1} \binom{h+n-1}{n-1}r! \binom{h+n}{r}\left(\tfrac12\right)^{h}.
\]

For each \emph{specific} $r$, this can be evaluated as a closed-form expression in $n$, and Maple can do it easily for small $r$, but as $r$ gets larger, it becomes harder and harder. There is no closed-form expression in $r$. Luckily, thanks to the Zeilberger algorithm~\cites{PWZ, Z}, one can get the following linear recurrence equation for $A(n,r)$ in $r$, where we abbreviate $\Cn\coloneqq n\binom{2n}{n}/4^n$:
\begin{align}\label{Arec}
A(n,r)= 2n A(n,r-1) + (r-1)(r-2)A(n,r-2) - 4n  \binom{2n-1}{r-2}(r-2)!\,\Cn  
\end{align}
subject to the initial conditions,
\begin{align}\label{Ainit}
A(n,1)=2n-2\Cn, \quad A(n,2)=4n^2-8n\Cn.  
\end{align}

This enables a very fast computation of $A(n,r)$ for many $r$. Once we have them, Maple can easily compute the (usual) moments
\[
\E[X^r]= \sum_{i=0}^{r} S(r,i)A(n,i),
\]
where $S(r,i)$ are the Stirling numbers of the second kind.

Now Maple can easily compute the \emph{central moments} where
$\mu\coloneqq\E[X]=A(n,1)$ (which is $2n-2\Cn$ by Equation~\eqref{Ainit}),
\[
\E[(X-\mu)^r]= \sum_{i=0}^{r} \binom{r}{i}(-\mu)^{r-i} \, \E[X^i].
\]

In particular the variance $\sigma^2\coloneqq \E[(X-\mu)^2]$. Finally it can take the limits of the scaled central moments
\[
\lim_{n \rightarrow \infty} \frac{\E[(X-\mu)^r]}{\sigma^r},
\]
and {\bf{surprise!}} They are exactly the same as the central scaled moments of $-|N(0,1)|$, that are easily computed by Maple. We verified it up to $200$ moments, but could have easily gone further. See the output file:
{\tt https://sites.math.rutgers.edu/\~{}zeilberg/tokhniot/oCoinToss4.txt}.

But in order to prove it for \emph{all} moments, we need some human ingenuity and paper-and-pencil good-old-traditional math.

\section{The Human Touch}\label{Shuman}

We now give a mathematical proof of the moment asymptotics found above,
using standard probabilistic methods and results. 
Similar results have presumably been shown several times earlier,
one recent example is \cite{VW}, but for completeness we give detailed proofs.

\subsection{Convergence in distribution, $n$ Heads or $n$ Tails}\label{Sor}
Let $(\xi_i)_0^\infty$ be an infinite sequence of independent fair coin
tosses, with $\xi_i=1$ representing ``Heads'' and $\xi_i=-1$ representing ``Tails''. Let $S_N\coloneqq \sum_{i=1}^N \xi_i$, for $N\ge0$. Thus $S_N$ is the total profit after $N$ fair coin tosses for a player betting on Heads.

Let $H_N$ and $T_N$ be the number of Heads and Tails, respectively, in the first $N$ tosses. Thus
\begin{align}\label{a01}
  H_N&=\frac{N+S_N}{2}, \qquad T_N=\frac{N-S_N}2.
\end{align}
We write for simplicity $\Xn$ for $X_1(n,n;\frac12)$., i.e.,
\begin{align}
\Xn\coloneqq\min\{N : H_N=n \text{ or } T_N=n \}
=\min\{N : \max{(H_N,T_N)}=n \}
.\end{align}
Note that 
\begin{align}\label{a2+}
H_N+T_N&=N, \\ \label{a2-}
H_N-T_N &= S_N.  
\end{align}
In particular, 
\begin{align}\label{a3}
\Xn=H_{\Xn}+T_{\Xn} \leqslant 2n.  
\end{align}
Furthermore, at time $\Xn$, one of $H_{\Xn}$ and $T_{\Xn}$ equals $n$ while the other is smaller. By Equation~\eqref{a2-}, the smaller one is $n-|S_{\Xn}|$,
and thus Equation~\eqref{a3} yields
\begin{align}\label{a4}
  \Xn=2n-|S_{\Xn}|.
\end{align}
Hence, the random variable $2n-\Xn$ that we are interested in is $|S_{\Xn}|$.
In particular, the centered variable 
\begin{align}\label{a4.5}
  \Xn-\E[\Xn] = -\bigl(|S_{\Xn}|-\E[S_{\Xn}]\bigr).
\end{align}

The idea to analyse $S_\Xn$ is that $\Xn\approx 2n$, and thus
$S_{\Xn}\approx S_{2n}$, which has a nice normal limit by the central limit
theorem. 
More precisely, we have the following, where $\dto$ denotes convergence in distribution. 
(This is included in \cite[Theorem 2.4(b)]{VW}, where a different proof is
given.)
\begin{lemma}\label{L1} 
As $\ntoo$,
\begin{align}\label{a11}
  \frac{S_\Xn}{\sqrt{n}} \dto N(0,2),
\end{align}
where $N(0,2)$ denotes a normal random variable with mean $0$ and variance $2$.
\end{lemma}

\begin{proof}  
One elegant way to prove this rigorously uses Donsker's theorem
on convergence of the entire process $(S_N)_{N=1}^{2n}$, after suitable
scaling, to a Brownian motion~\cite[Theorem 7.7.13]{Gut}.
But we choose instead to
proceed here by a related but somewhat more elementary approach.

We split the process of coin tosses into two phases; in the first we toss 
$n_1\coloneqq2n-\nii$ times, and in the second we proceed to the end.

By the central limit theorem, since $\E\xi_i=0$ and $\Var\xi_i=1$,
as $n\to\infty$
we have
${S_{n_1}}/{\sqrt{n_1}}\dto N(0,1)$ and thus
\begin{align}\label{a5}
  \frac{S_{n_1}}{\sqrt{n}}
= \sqrt{\frac{n_1}{n}}\cdot\frac{S_{n_1}}{\sqrt{n_1}}
=\bigl(\sqrt2+o(1)\bigr)\frac{S_{n_1}}{\sqrt{n_1}}
\dto N(0,2).
\end{align}
In particular, w.h.p. (with high probability, meaning with probability tending to 1 as $n\to\infty$),
$|S_{n_1}|<\nii=2n-n_1$, and thus
$H_{n_1},T_{n_1}<n$ by Equation~\eqref{a01}, 
so after $n_1$ tosses we have not yet reached the stopping time $\Xn$. 
We may thus assume this event, i.e., $\Xn> n_1$,
in the rest of the proof.

Let $n_2=\nii$, so $n_1+n_2=2n$. Then, by the assumption just made and Equation~\eqref{a3},  
\begin{align}
  \label{a6}
n_1 \leqslant \Xn \leqslant 2n=n_1+n_2.
\end{align}
Let 
\begin{align}\label{a7}
  S'_k\coloneqq S_{n_1+k}-S_{n_1} = \sum_{i=1}^k \xi_{n_1+i}.
\end{align}
By Kolmogorov's inequality~\cite[Theorem 3.1.6]{Gut}, for every $x>0$, 
\begin{align}\label{a8}
  \P\bigl(\max_{1\leqslant k\leqslant n_2}|S'_k| >x\bigr)
\leqslant \frac{\sum_{i=1}^{n_2}\Var(\xi_{i+n_1}) }{x^2}
=\frac{n_2}{x^2}
=\frac{\nii}{x^2}.
\end{align}
In particular, 
\begin{align}\label{a85}
    \P\bigl(\max_{1\leqslant k\leqslant n_2}|S'_k| >n^{0.4}\bigr) \to 0,
\end{align}
and thus w.h.p., recalling Equations~\eqref{a6} and \eqref{a7},
\begin{align}
\label{a9}
|S_\Xn-S_{n_1}|
=|S'_{\Xn-n_1}|
\leqslant  \max_{1\leqslant k\leqslant n_2}|S'_k|  
\leqslant n^{0.4}.
\end{align}
Hence, with $\pto$ denoting convergence in probability,
\begin{align}\label{a10}
  \frac{S_\Xn-S_{n_1}}{\sqrt{n}} \pto 0,
\end{align}
which together with Equation~\eqref{a5} yields, 
by the Cram\'er--Slutsky theorem \cite[Theorem 5.11.4]{Gut},
\begin{align}\label{b1}
  \frac{S_\Xn}{\sqrt{n}} 
=
  \frac{S_\Xn-S_{n_1}}{\sqrt{n}} 
+
  \frac{S_{n_1}}{\sqrt{n}} \dto N(0,2).
\end{align}
which shows~\eqref{a11}.
\end{proof}

We can now easily show that the asymptotic distribution of $X_1(n)$ is as
found empirically above.
\begin{theorem}\label{T1}
Let $Z\sim N(0,1)$ denote a standard normal variable.
Then, as $n\to\infty$,
  \begin{align}
    \label{t1a}
  \frac{\Xn-2n}{\sqrt{n}} \dto  -\sqrt2|Z|,
  \end{align}
and, for the centered variables,
  \begin{align}
    \label{t1b}
  \frac{\Xn-\E[\Xn]}{\sqrt{n}} \dto  -\sqrt2\bigl(|Z|-\E|Z|\bigr).
  \end{align}
\end{theorem}

\begin{proof}
We can write Equation~\eqref{a11} as
\begin{align}
\label{a12}
  \frac{S_\Xn}{\sqrt{n}} \overset{\mathrm{d}}{\longrightarrow} \sqrt2Z.
\end{align}
Hence, by the continuous mapping theorem~\cite[Theorem 5.10.4]{Gut},
\begin{align}\label{b2}
  \frac{|S_\Xn|}{\sqrt{n}} \dto \sqrt2\,|Z|,
\end{align}
and~\eqref{t1a} follows by Equation~\eqref{a4}.
Finally,~\eqref{t1b} follows from~\eqref{t1a} and
  \begin{align}\label{b1a}
  \frac{2n-\E[\Xn]}{\sqrt{n}} 
\longrightarrow
\sqrt2 \E|Z| =2/\sqrt\pi,
  \end{align}
which follows from Equation~\eqref{Ainit} which gives $2n-\E[\Xn]=2\Cn$,
or from Theorem~\ref{Tmom} below.
\end{proof}

We postpone further discussion of convergence of moments until Section~\ref{Smom}.

\subsection{Convergence in distribution, $n$ Heads and $n$ Tails}\label{Sand}

We can argue similarly with 
\begin{align}
\label{b02}
  \Yn\coloneqq\min\{N\; : H_N\geqslant n \text{ and  } T_N \geqslant n \} =\min\{N: \min{(H_N,T_N)}=n \}
.\end{align}
Now, $\Yn\geqslant 2n$. Furthermore, one of $H_\Yn$ and $T_\Yn$ is $n$, and the
other is, by Equation~\eqref{a2-}, $n+|S_\Yn|$. Consequently, by Equation~\eqref{a2+},
\begin{align}
  \label{b4}
\Yn=H_\Yn+T_\Yn
=2n+|S_\Yn|.
\end{align}
In analogy with Lemma~\ref{L1} we have:
\begin{lemma}\label{L2} 
As $\ntoo$,
\begin{align}\label{b11}
  \frac{S_\Yn}{\sqrt{n}} \dto N(0,2).
\end{align}
\end{lemma}

\begin{proof}
Let $n_2=\nii$ as above.
It follows from the central limit theorem, similarly to Equation~\eqref{a5},
that
\begin{align}
\label{b5}
  \frac{S_{2n+n_2}}{\sqrt{n}} \dto  N(0,2).
\end{align}
and it follows  that w.h.p. $|S_{2n+n_2}|<n_2$ and thus, by Equation~\eqref{a01},
$H_{2n+n_2},T_{2n+n_2}>n$. Hence, w.h.p.,
\begin{align}\label{b6}
2n \leqslant \Yn < 2n+n_2.
\end{align}
We use Kolmogorov's inequality \eqref{a8} again, but now for $1\leqslant k\leqslant 2n_2$,
and obtain from Equation~\eqref{b6} as in \eqref{a85}--\eqref{a10},
\begin{align}\label{b10}
  \frac{S_\Yn-S_{n_1}}{\sqrt{n}} \dto 0,
\end{align}
and thus~\eqref{b11} follows from  Equation~\eqref{a5} as in~\eqref{b1}.
\end{proof}

Hence we obtain, in analogy to Theorem~\ref{T1}:
\begin{theorem}\label{T2}
  We have, with notation defined above,
as $n\to\infty$,
  \begin{align}
    \label{t2a}
  \frac{\Yn-2n}{\sqrt{n}} \dto \sqrt2|Z|,
  \end{align}
and, for the centered variables,
  \begin{align}
    \label{t2b}
  \frac{\Yn-\E\Yn}{\sqrt{n}} \dto \sqrt2\bigl(|Z|-\E|Z|\bigr).
  \end{align}
\end{theorem}

\begin{proof}
Almost identical to the proof of Theorem~\ref{T1}, now using Equation~\eqref{b4} and Lemma~\ref{L2}; 
note also that 
\begin{align}\label{ql3}
\E[\Xn]+\E[\Yn]=L_1(n,n;\tfrac12)+L_2(n,n;\tfrac12)=4n,  
\end{align}
by Equation~\eqref{l1l2}, and thus $\E[\Yn]-2n=2n-\E[\Xn]$.
\end{proof}

\begin{remark}
Thus $\Xn$ and $\Yn$ have, apart from a sign, the same (centered) asymptotic distribution. Moreover, it is really ``the same'' $Z$ in
Theorem~\ref{T1} and~\ref{T2}: it follows from Equations~\eqref{a4}, \eqref{b4}, \eqref{a10}, and \eqref{b10}
that
\begin{align}\label{xy1}
\frac{\Xn-2n}{\sqrt n}+ \frac{\Yn-2n}{\sqrt n}
=\frac{|S_\Yn|-|S_\Xn|}{\sqrt n} \dto 0.
  \end{align}
In fact, an extension of the arguments above shows that 
$(\Yn-2n)-(2n-\Xn)$ 
is of order $n^{1/4}$.
More precisely, we have the following result for this difference,
showing that it has an asymptotic distribution that is a mixture of normal
distributions with different variances.
\end{remark}

\begin{theorem}
As $n\to\infty$,
  \begin{align}\label{XY}
\frac{(\Yn-2n)-(2n-\Xn)}{n^{1/4}} 
=\frac{\Xn+\Yn-4n}{n^{1/4}} \dto 2^{3/4}|Z|^{1/2}W,
\end{align}
where $W$ and $Z$ are independent random variables with the standard normal 
distribution $N(0,1)$.
\end{theorem}

\begin{proof}[Sketch of proof]
Condition on $\Xn=2n-x\sqrt n$, for some $x>0$.
We see from \eqref{xy1} that then $\Yn-2n\approx 2n-\Xn=x\sqrt n$, and thus
$\Yn-\Xn\approx 2x\sqrt n =2(2n-\Xn)$. 
It follows, by arguments as in the proofs above, that, conditioned on $\Xn$,
\begin{align}\label{xy2}
  \frac{S_\Yn-S_\Xn}{\sqrt{2(2n-\Xn)}}
\dto W \in N(0,1).
\end{align}
Furthermore, it is easy to see that
$S_\Xn$ and $S_\Yn$ have the same sign (depending on whether we reach $n$
Heads or $n$ Tails first); consequently,
\eqref{xy2} implies that, still conditioned on $\Xn$,
\begin{align}\label{xy3}
  \frac{|S_\Yn|-|S_\Xn|}{\sqrt{2(2n-\Xn)}}
\dto W \in N(0,1).
\end{align}
This implies that \eqref{xy3} holds also unconditionally, with $W$ independent
of all $\Xn$. 
Hence, using also Equations~\eqref{a4}, \eqref{b4}, and \eqref{t1a},
\begin{align}\label{xy4}
\frac{\Xn-2n}{n^{1/4}}+ \frac{\Yn-2n}{n^{1/4}}
&=
  \frac{|S_\Yn|-|S_\Xn|}{n^{1/4}}
=
  \frac{|S_\Yn|-|S_\Xn|}{\sqrt{2(2n-\Xn)}}\cdot 
\sqrt{\frac{2(2n-\Xn)}{n^{1/2}}}
\notag\\&
\dto W\sqrt{2\sqrt2 |Z|},
\end{align}
with $Z\in N(0,1)$ independent of $W$, which proves \eqref{XY}.
\end{proof}

\subsection{Convergence of moments}\label{Smom}

The results in Theorem~\ref{T1} and \ref{T2} are convergence in
distribution, and as always, this does not by itself imply convergence of
moments.
In this case, as in many others,
it is easy to give supplementary arguments 
showing that the moments converge, as found empirically in 
Section \ref{sec:Explicit-exprs}.

\begin{theorem}\label{Tmom}
  We have convergence of all moments (both ordinary and absolute)
in~\eqref{t1a}, \eqref{t1b}, \eqref{t2a}, and \eqref{t2b}.
\end{theorem}

\begin{proof}
Consider the number of tosses until reaching $n$ Heads, or $n$ Tails,
separately: 
\begin{align}\label{c1}
\nu_H(n)\coloneqq\inf\{k : H_k\geqslant n\},
\qquad
\nu_T(n)\coloneqq\inf\{k : T_k\geqslant n\}. 
\end{align}
Then 
\begin{align}\label{c2h}
  \Xn &= \nu_H(n) \land \nu_T(n),
\\
  \Yn &= \nu_H(n) \lor \nu_T(n).\label{c2t}
\end{align}
Note that $\nu_H(n)$ and $\nu_T(n)$ have the same distribution,
which is negative binomial with the simple probability generation function~\eqref{negbin}  mentioned in Section~\ref{Spreface}.
However, they are dependent, 
so the representation~\eqref{c2h}--\eqref{c2t} does not 
tell us the  distribution of $\Xn$ and $\Yn$,
but it is nevertheless very helpful to obtain useful estimates.

Let $r>0$. It is well-known from renewal theory, 
see e.g.~\cite[Theorem 3.7.4(ii)]{Gut:SRW}, that
the sequence of random variables
\begin{align*}
 \left\lvert\frac{\nu_H(n)-2n}{\sqrt n}\right\rvert^r,
\qquad n \geqslant 1,
\end{align*}
is \emph{uniformly integrable}. 
(See e.g.~\cite[Section 5.4]{Gut} for the definition.)
The same is true for $\nu_T(n)$, since it has the same distribution as
$\nu_H(n)$, and then it follows from Equations~\eqref{c2h}--\eqref{c2t} that 
\begin{align}\label{c4}
 \left\lvert\frac{\Xn-2n}{\sqrt n}\right\rvert^r
\qquad \text{and}\qquad
\left\lvert\frac{\Yn-2n}{\sqrt n}\right\rvert^r,
\qquad n \geqslant 1,
\end{align}
also are uniformly integrable.

This implies that all moments converge in~\eqref{t1a} and \eqref{t2a},
see e.g.~\cite[Theorem 5.5.9]{Gut}.
In particular, this shows that $(\E[\Xn]-2n)/\sqrt n$ converges to
$\E[\sqrt2|Z|]$, so $\E[\Xn]=2n+O(\sqrt n)$,
as also was seen in Section~\ref{sec:Explicit-exprs}.
Similarly, or by \eqref{l1l2}, $\E[\Yn]=2n+O(\sqrt n)$. 
This and \eqref{c4} implies that
\begin{align}
\left\lvert\frac{\Xn-\E[\Xn]}{\sqrt n}\right\rvert^r
\qquad \text{and}\qquad
\left\lvert\frac{\Yn-\E[\Yn]}{\sqrt n}\right\rvert^r,
\qquad n \geqslant 1,
\end{align}
also are uniformly integrable. 
Hence we have moment convergence in~\eqref{t1b} and \eqref{t2b} too.
\end{proof}

\section{Proof of \eqref{l1ab}--\eqref{l2ab}}\label{Spfl1ab}

We give here a probabilistic proof of \eqref{l1ab}--\eqref{l2ab}.
Note that \eqref{l1l2} yields
\begin{align}
  L_1\Bigl(an,bn;\frac{a}{a+b}\Bigr)
+
  L_2\Bigl(an,bn;\frac{a}{a+b}\Bigr)
=\frac{an}{a/(a+b)}+\frac{bn}{b/(a+b)}
=2(a+b)n,
\end{align}
and thus Equations~\eqref{l1ab} and \eqref{l2ab} are equivalent, so it suffices to consider Equation~\eqref{l1ab}.

We use the standard notation, for a random variable $\cX$ and an event $\cE$, $\E[\cX;\cE]\coloneqq\E[\cX\cdot\boldsymbol1\{\cE\}]$.
We use also the following well-known result.

\begin{lemma}\label{L3}
  Let $\cX\in\mathrm{Bin}(N,p)$ be a binomial random variable. Then, for
  $0\le k< N$,
  \begin{align}
    \label{l3}
\E[Np-\cX;\,\cX\le k]
=\sum_{i=0}^k\binom{N}{i}p^i(1-p)^{N-i}(Np-i)
=\frac{N!}{(N-k-1)!\,k!}p^{k+1}(1-p)^{N-k}.
  \end{align}
\end{lemma}
\begin{proof}
  The first equality in \eqref{l3} is just the definition of expectation;
the second follows by an easy induction on $k$.
\end{proof}

\begin{proof}[Proof of Equation~\eqref{l1ab}]
  Fix positive integers $a$ and $b$, and consider a biased coin with
  $\P(\text{Heads})=\frac{a}{a+b}$. 
We reuse the notation in Section~\ref{Shuman} 
(where we studied  the case $a=b=1$),
with some modifications. In particular, now $\Xn\coloneqq X_1(an,bn;\frac{a}{a+b})$.
We now define $\xi_i$ by
\begin{align}
  \xi_i\coloneqq
  \begin{cases}
    b, & \text{if toss $i$ is Heads},
\\
    -a, & \text{if toss $i$ is Tails}.
  \end{cases}  
\end{align}
Then $\E[\xi_i]=0$, so $\xi_i$ still represents a fair game.
As before, let $S_N\coloneqq\sum_{i=1}^N\xi_i$ be the total profit after $N$ coin
tosses. 
Let again $H_N$ and $T_N$ be the number of Heads and Tails in the $N$ first
coin tosses. We have, for every $N\ge0$,
\begin{align}
  N&=H_N+T_N, \label{k1}
\\
S_N&=bH_N-aT_N.\label{k2}
\end{align}

We have either $H_\Xn=an$ and $T_\Xn<bn$ (``{Heads win}'')
or
$H_\Xn<an$ and $T_\Xn=bn$ (``{Tails win}'').
In both cases, Equation~\eqref{k1} yields $\Xn=H_\Xn+T_\Xn<(a+b)n$. Hence, if Heads win, then $H_{(a+b)n}\ge H_\Xn=an$ and thus (using Equation~\eqref{k1} again) $T_{(a+b)n}\le bn$, and consequently, by Equation~\eqref{k2}, $S_{(a+b)n}\ge0$.
Conversely, if Tails win, then $S_{(a+b)n}\le 0$.

Now condition on the event that Heads win, and more precisely $T_\Xn=t$ for
some given $t<bn$. 
Then $\Xn=H_\Xn+T_\Xn=an+t$. 
Continue to toss the coin after $\Xn$; these tosses are independent of what 
happened earlier, and have means $\E\xi_i=0$ as noted above, 
and thus the conditional expectation of
$S_{(a+b)n}-S_\Xn=\sum_{i=an+t+1}^{(a+b)n}\xi_i$ is $0$.
Consequently,
\begin{align}\label{k3}
  \E[S_{(a+b)n};T_\Xn=t] - \E[S_\Xn;T_\Xn=t]
=  \E[S_{(a+b)n}-S_\Xn;T_\Xn=t]
=0.
\end{align}
Summing over all $t<bn$ we obtain, using the comments above, 
\begin{align}\label{k4}
  \E[S_\Xn;\,\text{Heads win}]
=
  \E[S_{(a+b)n};\,\text{Heads win}]
=
  \E[S_{(a+b)n};\,S_{(a+b)n}\ge0].
\end{align}
We analyze the two sides of \eqref{k4} separately.
For the left-hand side, we note that when Heads win, so $H_\Xn=an$, we have
by Equations~\eqref{k2} and \eqref{k1},
\begin{align}\label{k5}
  S_\Xn&=bH_\Xn-aT_\Xn
= (a+b)H_\Xn-a(H_\Xn+T_\Xn) = (a+b)an-a\Xn
\notag\\&
= a[(a+b)n-\Xn].
\end{align}
Hence,
\begin{align}\label{k6}
  \E[S_\Xn;\,\text{Heads win}]
=
 a \E[(a+b)n-\Xn;\,\text{Heads win}].
\end{align}
For the right-hand side, we have, again by Equations~\eqref{k2} and \eqref{k1},
\begin{align}
  \label{k7}
S_{(a+b)n}=bH_{(a+b)n}-aT_{(a+b)n}
=b(H_{(a+b)n}+T_{(a+b)n})-(a+b)T_{(a+b)n}
=(a+b)(bn-T_{(a+b)n}).
\end{align}
Hence, using also Lemma \ref{L3} (with $k=bn$), since 
$T_{(a+b)n}\in\mathrm{Bin}\bigl((a+b)n,\frac{b}{a+b}\bigr)$,
\begin{align}  \label{k8}
\E[S_{(a+b)n};\,S_{(a+b)n}\ge0]
=(a+b)\E[bn-T_{(a+b)n};\,bn-T_{(a+b)n}\ge0]
=\frac{((a+b)n)!}{(an-1)!\,(bn)!}\cdot\frac{b^{bn+1}a^{an}}{(a+b)^{(a+b)n}}.
\end{align}
Combining Equations~\eqref{k6}, \eqref{k4}, and \eqref{k8} yields
\begin{align}  \label{k9}
\E[(a+b)n-\Xn;\,\text{Heads win}]
=a^{-1} \E[S_{(a+b)n};\,S_{(a+b)n}\ge0]
=n\frac{((a+b)n)!}{(an)!\,(bn)!}\cdot\frac{b^{bn+1}a^{an}}{(a+b)^{(a+b)n}}.
\end{align}
Interchanging Heads and Tails (and thus $a$ and $b$) in \eqref{k9} yields
\begin{align}  \label{k9t}
\E[(a+b)n-\Xn;\,\text{Tails win}]
=n\frac{((a+b)n)!}{(an)!\,(bn)!}\cdot\frac{a^{an+1}b^{bn}}{(a+b)^{(a+b)n}}.
\end{align}
Finally, summing Equations~\eqref{k9} and \eqref{k9t} yields
\begin{align}  \label{k10}
\E[(a+b)n-\Xn]
=n(a+b)\frac{((a+b)n)!}{(an)!\,(bn)!}\cdot\frac{a^{an}b^{bn}}{(a+b)^{(a+b)n}},
\end{align}
which is Equation~\eqref{l1ab}.
\end{proof}

\section*{Acknowledgements}
S.~Janson was supported by the Knut and Alice Wallenberg Foundation and the Swedish Research Council. L.~Martinez was supported by the NSF Graduate Research Fellowship Program under Grant No. 2233066.

\end{document}